\documentclass{amsart}

\usepackage[all]{xypic}
\usepackage{tikz}
\usetikzlibrary{arrows}
\usetikzlibrary{decorations.markings}

\usepackage{graphicx}
\usepackage{bm}
\usepackage{epsf}
\usepackage{verbatim} 
\usepackage{amsmath}
\usepackage{amsfonts}
\usepackage{amssymb}
\usepackage{mathrsfs}
\usepackage{amsthm}
\usepackage{newlfont}

 \newtheorem{thm}{Theorem}

 \newtheorem{lem}[thm]{Lemma}
 \newtheorem{prop}[thm]{Proposition}
 \theoremstyle{definition}
 \newtheorem{defn}[thm]{Definition}
 \theoremstyle{remark}
 \newtheorem{rem}[thm]{Remark}
 \theoremstyle{definition}

\theoremstyle{definition}

 \theoremstyle{definition}
 
 \numberwithin{equation}{section}
 \numberwithin{thm}{section}

 \newcommand{\odd}{\mathrm{odd}}
 \newcommand{\disc}{\mathrm{disc}}

 \newcommand{\Hom}{\mathrm{Hom}}
 \newcommand{\Ext}{\mathrm{Ext}}
 \newcommand{\Spec}{\mathrm{Spec}}

 \newcommand{\Aut}{\mathrm{Aut}}
 \newcommand{\End}{\mathrm{End}}

 \newcommand{\Gal}{\mathrm{Gal}}
 \newcommand{\GL}{\mathrm{GL}}

 \newcommand{\coker}{\mathrm{coker}}

 \newcommand{\tor}{\mathrm{tor}}

 \newcommand{\new}{\mathrm{new}}

 \renewcommand{\mod}{\mathrm{mod}}

 \newcommand{\fm}{\mathfrak m}

 \newcommand{\cO}{\mathcal{O}}

 \newcommand{\cC}{\mathcal{C}}
 \newcommand{\cA}{\mathcal{A}}
 \newcommand{\cB}{\mathcal{B}}

 \newcommand{\cE}{\mathcal{E}}

 \newcommand{\gm}[1]{\mathbb{G}_{m,#1}}

 \newcommand{\F}{\mathbb{F}}
 \newcommand{\Q}{\mathbb{Q}}
 \newcommand{\Z}{\mathbb{Z}}

 \newcommand{\T}{\mathbb{T}}

 \newcommand{\To}{\longrightarrow}

 \newcommand{\tT}{\widetilde{\mathbb{T}}}

 \newcommand{\G}{\Gamma}

\begin{document}

\title[On Ribet's isogeny for $J_0(65)$]{On Ribet's isogeny for $J_0(65)$}

\author{Krzysztof Klosin}
\address{Department of Mathematics, 
Queens College, 
City University of New York, 
65-30 Kissena Blvd
Flushing, NY 11367, USA}
\email{kklosin@qc.cuny.edu} 

\author{Mihran Papikian}
\address{Department of Mathematics, Pennsylvania State University, University Park, PA 16802, USA}
\email{papikian@psu.edu}

\thanks{The first author was supported by the Young Investigator Grant \#H98230-16-1-0129 from the National Security Agency, and by a PSC-CUNY award jointly funded by the Professional Staff Congress and the City University of New York.}

\thanks{The second author was partially supported by grants from the Simons Foundation (245676) and the National Security Agency 
(H98230-15-1-0008).} 



\begin{abstract} Let $J^{65}$ be the Jacobian of the Shimura curve attached to the indefinite quaternion algebra 
over $\Q$ of discriminant $65$. We study the isogenies $J_0(65)\to J^{65}$ 
defined over $\Q$, whose existence was proved by Ribet. We prove that there is an isogeny 
whose kernel is supported on the Eisenstein maximal ideals of the Hecke algebra acting on $J_0(65)$, 
and moreover the odd part of the kernel is generated by a cuspidal divisor of order $7$, as is predicted by 
a conjecture of Ogg. 
\end{abstract}

\maketitle


\section{Introduction} 

Let $N$ be a product of an even number of distinct primes. Let $J_0(N)$ 
be the Jacobian of the modular curve $X_0(N)$. In \cite{RibetIsogeny}, 
Ribet proved the existence of an isogeny defined over $\Q$ between the ``new'' 
part $J_0(N)^\new$ of $J_0(N)$ and the Jacobian $J^N$ 
of the Shimura curve $X^N$ attached to a maximal order in 
the indefinite quaternion algebra over $\Q$ of discriminant $N$. 
Although there are no morphisms $X_0(N)\to X^N$ defined over $\Q$, 
Ribet showed 
that the $\Q_\ell$-adic Tate modules of $J_0(N)^\new$ and $J^N$ are 
isomorphic as $\Gal(\overline{\Q}/\Q)$-modules, where $\ell$ is an arbitrary prime number; this is a 
consequence of 
a correspondence between automorphic forms on $\GL(2)$ and automorphic forms on the multiplicative group of a 
quaternion algebra. The existence of the isogeny $J_0(N)^\new\to J^N$ defined 
over $\Q$ then follows from a special case of Tate's isogeny conjecture for abelian varieties over number fields,  
also proved in \cite{RibetIsogeny} (the general case of Tate's conjecture was proved a few years later by Faltings). 
Unfortunately, Ribet's argument provides no information about the isogenies $J_0(N)^\new\to J^N$ 
beyond their existence. 

In \cite{Ogg}, Ogg made an explicit conjecture about the kernel of Ribet's 
isogeny when $N=pq$ is a product of two distinct primes and $p=2,3,5,7,13$: 
the conjecture predicts that there is an isogeny $J_0(N)^\new\to J^N$ of minimal degree 
whose kernel is a specific group arising from the cuspidal divisor subgroup of $J_0(N)$. 
Note that $p=2,3,5,7,13$ are exactly the primes for which $J_0(pq)$ has purely toric reduction 
at $q$. This fact is crucial for the calculations used by Ogg to come up with his conjecture;  
the underlying idea is that the knowledge of the group of connected components of the N\'eron 
models of $J_0(N)^\new$ and $J^N$ at $q$ yields restrictions on the isogenies between them. 
Ogg's conjecture remains open except for the special cases when $J^N$ has dimension $\leq 3$. 

When $\dim(J^N)=1$, equiv. $N=2\cdot 7$, $3\cdot 5$, $3\cdot 7$, $3\cdot 11$, $2\cdot 17$, $J^N$ 
is an elliptic curve over $\Q$ which is uniquely determined by its component groups at $p$ and $q$, 
and $J_0(N)^\new$ is the optimal elliptic curve of conductor $N$. Then one easily checks 
Ogg's conjecture using Cremona's tables \cite{Cremona}. 
In general, the orders of component groups of $J^N$ can be computed using 
Brandt matrices \cite{JordanLivne}, which is relatively easy to do with the help of a computer program such as \texttt{Magma}. 

When $\dim(J^N)=2$, equiv. $N=2\cdot 13$, $2\cdot 19$, $2\cdot 29$, Ogg's conjecture 
is verified in \cite{GoRo}. In this case, the proof is based on the fact that $X^N$ is bielliptic 
and the lattices of $J_0(N)^\new$ and $J^N$ can be computed through their elliptic quotients. 

When $\dim(J^N)=3$, 
equiv. $N=2\cdot 31$, $2\cdot 41$, $2\cdot 47$, $3\cdot 13$, $3\cdot 17$, $3\cdot 19$, $3\cdot 23$, $5\cdot7$, $5\cdot 11$, 
Ogg's conjecture is verified in \cite{GM}. In this case, $X^N$ is always hyperelliptic. 
By utilizing this fact, Gonz\'alez and Molina explicitly compute the equation for each $X^N$. 
Then they obtain a basis of regular differentials for $X^N$ from these equations to produce a 
period matrix for $J^N$. The period matrix of $J_0(N)^\new$ can be computed using cusp forms with 
rational $q$-expansions. The problem then reduces to comparing the period matrices 
of appropriate quotients of $J_0(N)^\new$ with the period matrix of $J^N$. 

The goal of this paper is to study Ribet's isogeny for $N=5\cdot 13=65$. In this case, $\dim(J^N)=5$ 
and $X^N$ is \textit{not} hyperelliptic; cf. \cite{Michon}. Our approach to the study of Ribet isogenies 
is completely different from that in \cite{GoRo} and \cite{GM}, and crucially relies on the Hecke 
equivariance of such isogenies. In this approach we need to know very little about $X^N$ or $J^N$; 
we only need to know the orders of component groups of $J^N$, which, as we mentioned, are easy to compute, 
and in fact were already computed in \cite{Ogg}. The difficulty shifts to the study of the structure of the Hecke 
algebra and its action on $J_0(N)$.

Let $\T(N):=\Z[T_2, T_3, \dots]$ 
be the $\Z$-algebra generated by the Hecke operators $T_n$ 
acting on be the space $S_2(N)$ of weight $2$ cups forms on $\G_0(N)$. This 
algebra is isomorphic to the subalgebra of $\End(J_0(N))$ generated 
by $T_n$ acting as correspondences on $X_0(N)$. When $N=65$, we have $J_0(N)^\new=J_0(N)$, 
so there is a Ribet isogeny 
$$
\pi: J_0(N)\to J^N.
$$ 
$\T(N)$ also naturally acts on $J^N$ and 
$\pi$ is $\T(N)$-equivariant. This equivariance is implicit in Ribet's proof \cite{RibetIsogeny}; see 
also \cite[Cor. 2.4]{Helm}. 

From now on we assume $N=65$. To simplify the notation, we denote $\T:=\T(N)$, $J:=J_0(N)$, 
$J':=J^N$, $G_\Q:=\Gal(\overline{\Q}/\Q)$. Given a finite abelian group $H$, we denote by $H_p$ 
its $p$-primary component ($p$ is a prime number), and by $H_\odd$ its maximal subgroup of odd order,  
so that $H\cong H_2\times H_\odd$. 
Since the endomorphisms of $J$ induced by Hecke operators 
are defined over $\Q$, the actions of $\T$ and $G_\Q$ on $J$ commute with each other. 
Thus, $\ker(\pi)$ is a $\T[G_\Q]$-submodule of $J$. We show that if 
the kernel of an isogeny from $J$ to another abelian variety is a $\T[G_\Q]$-module, then, 
up to endomorphisms of $J$, the kernel is supported on the Eisenstein maximal 
ideals of $\T$. We then classify all $\T[G_\Q]$-submodules of $J$ of odd order supported 
on the Eisenstein maximal ideals. This leads to the following theorem, which is the main result of the paper:
\begin{thm}\label{thmIntro}
There is a Ribet isogeny $\pi:J\to J'$ such that $\ker(\pi)_\odd\cong \Z/7\Z$ is the $7$-primary 
component of the cuspidal divisor group of $J$. 
\end{thm}

Ogg's conjecture in this case predicts that in fact $\ker(\pi)=\Z/7\Z$. There is a unique Eisenstein maximal ideal 
$\fm_2\lhd \T$ of residue characteristic $2$. In principle, it should be 
possible to extend our analysis to finite $\T[G_\Q]$-submodules of $J$ supported on $\fm_2$ 
to show that $\ker(\pi)_2=0$. But there are several technical difficulties which at present we 
are not able to overcome: these stem from the fact that $\fm_2$ is a prime of fusion, $\T_{\fm_2}$ 
is not Gorenstein, and the groups of rational points of reductions of $J$ usually have large $2$-primary 
components. 

Our strategy can be applied also to cases when $\dim(J^N)=3$, which leads to results similar to 
Theorem \ref{thmIntro}, at least when $J_0(N)^\new=J_0(N)$ (equiv. $N=3\cdot 13, 5\cdot 7$); see 
Remarks \ref{rem4.9} and  \ref{rem4.10}. 

\begin{rem} Given a prime $\ell$, if $H:=(J_0(N)^\new(\Q)_\tor)_\ell\neq 0$ but $(J^N(\Q)_\tor)_\ell=0$, 
then obviously $H\subset \ker(\pi)$ for any Ribet isogeny $\pi: J_0(N)^\new\to J^N$.  
For an odd prime $\ell$, in \cite{YooBLMS}, Yoo gives sufficient conditions for the non-existence of rational 
points of order $\ell$ on $J^N$, when $N=pq$ is a product of two distinct primes. This then can be used 
to find non-trivial subgroups of the kernels of Ribet isogenies; see \cite[Thm. 1.3]{YooBLMS}. 
In the case when $N=65$, Yoo's theorem 
implies that $\Z/7\Z\subset \ker(\pi)$. 
\end{rem}


\section{N\'eron models} 

In this section we recall some terminology and facts from the theory of N\'eron models. 
Let $R$ be a complete discrete valuation ring,
with fraction field $K$ and residue field $k$.
Let $A$ be an abelian variety over $K$. Denote by $\cA$ its N\'eron
model over $R$ and denote by $\cA_k^0$ the connected component of the
identity of the special fiber $\cA_k$ of $A$. There is an exact
sequence
$$
0\to \cA_k^0\to \cA_k\to \Phi_A\to 0,
$$
where $\Phi_A$ is a finite (abelian) group called the
\textit{component group of $A$}. We say that $A$ has
\textit{semi-abelian reduction} if $\cA_k^0$ is an extension of an
abelian variety $A_k'$ by an affine algebraic torus $T_A$ over $k$
(cf. \cite[p. 181]{NM}):
$$
0\to T_A\to \cA_k^0\to A_k'\to 0.
$$
We say that $A$ has \textit{good reduction}, if $\cA_k^0=A_k'$ (in this case, we also have $\cA_k=\cA_k^0$); 
we say that $A$ has (purely) \textit{toric reduction} if $\cA_k^0=T_A$. The
\textit{character group}
\begin{equation}\label{eqM}
M_A:=\Hom((T_{A})_{\bar{k}}, \gm{\bar{k}})
\end{equation}
is a free abelian group contravariantly associated to $A$.

Let $K'$ be a finite unramified extension of $K$, with ring of integers $R'$ and 
residue field $k'$. By the fundamental property of N\'eron models, we have an isomorphism of groups 
$A(K')\cong\cA(R')$, which defines a canonical reduction map 
\begin{equation}\label{eqCanRed}
A(K')\to \cA_k(k'). 
\end{equation}
Composing \eqref{eqCanRed} with  $\cA_k\to \Phi_A$, we get a homomorphism 
\begin{equation}\label{eqCanRedPhi}
A(K')\to \Phi_A. 
\end{equation}

\begin{prop}\label{propKatz} Let $K'$ be a finite unramified extension of $K$. 
Let $H\subset A(K')$ be a finite subgroup. 
Assume that either $\# H$ is coprime to the characteristic $p$ of $k$, or that $K$ has characteristic $0$ and its absolute ramification 
index is $< p-1$. Then \eqref{eqCanRed} defines an injection $H\hookrightarrow \cA_k(k')$. 
\end{prop}
\begin{proof}
See \cite[p. 502]{Katz} and \cite[Prop. 7.3/3]{NM}. 
\end{proof}

Let $\varphi: A\to B$ be an isogeny defined over $K$. By the N\'eron mapping property, $\varphi$ extends to a
morphism $\varphi: \cA\to \cB$ of the N\'eron models. On the special
fibers we get a homomorphism $\varphi_k: \cA_k\to \cB_k$, which induces
an isogeny $\varphi_k^0:\cA_k^0\to \cB_k^0$; \cite[Cor. 7.3/7]{NM}. This implies 
that $B$ has semi-abelian (resp. toric) reduction if $A$ has semi-abelian (resp. toric) reduction. 
The isogeny $\varphi_k^0$
restricts to an isogeny $\varphi_t:T_A\to T_B$, which corresponds to
an injective homomorphisms of character groups $\varphi^\ast: M_B\to
M_A$ with finite cokernel. We also get a natural homomorphism
$\varphi_\Phi: \Phi_A\to \Phi_B$. 

Denote by $\hat{A}$ the dual abelian variety of $A$. 
Let $\hat{\varphi}:\hat{B}\to \hat{A}$ be the isogeny dual to $\varphi$. 
Assume $A$ has semi-abelian reduction. 
In \cite{SGA7}, Grothendieck defined a non-degenerate pairing
$u_A:M_A\times M_{\hat{A}}\to \Z$ (called \textit{monodromy
pairing}) with nice functorial properties, which induces an exact
sequence
\begin{equation}\label{eqGrothM}
0\to M_{\hat{A}}\xrightarrow{u_A} \Hom(M_A, \Z)\to \Phi_A\to 0.
\end{equation}
Using \eqref{eqGrothM}, one obtains a commutative diagram with exact rows
(cf. \cite[p. 8]{RibetSTN}):
$$
\xymatrix{ 0 \ar[r] & M_{\hat{A}} \ar[r]
\ar[d]^-{\hat{\varphi}^\ast} & \Hom(M_A, \Z)
\ar[d]^-{\Hom(\varphi^\ast, \Z)} \ar[r]
& \Phi_A \ar[d]^-{\varphi_\Phi} \ar[r] & 0 \\
0 \ar[r] & M_{\hat{B}} \ar[r] & \Hom(M_B, \Z) \ar[r] & \Phi_B \ar[r]
& 0.}
$$
From this diagram we get the exact sequence
\begin{equation}\label{eq2'}
0\to \ker(\varphi_\Phi)\to
M_{\hat{B}}/\hat{\varphi}^\ast(M_{\hat{A}})\to
\Ext^1_\Z(M_A/\varphi^\ast(M_B), \Z)\to \coker(\varphi_\Phi)\to 0.
\end{equation}
Since 
$$
\Ext^1_\Z(M_A/\varphi^\ast(M_B), \Z)\cong
\Hom(M_A/\varphi^\ast(M_B), \Q/\Z)=: (M_A/\varphi^\ast(M_B))^\vee,
$$
we can rewrite \eqref{eq2'} as 
\begin{equation}\label{eq2}
0\to \ker(\varphi_\Phi)\to
M_{\hat{B}}/\hat{\varphi}^\ast(M_{\hat{A}})\to
(M_A/\varphi^\ast(M_B))^\vee\to \coker(\varphi_\Phi)\to 0.
\end{equation}

Note that $M_A/\varphi^\ast(M_B)\cong \Hom(\ker(\varphi_t), \gm{k})$. On the other hand, 
$\ker(\varphi_t)$ can be canonically identified with a subgroup scheme of $H:=\ker(\varphi)$; cf. \cite[p. 762]{CS}.  
Therefore, $\# M_A/\varphi^\ast(M_B)$ divides $\#H$. Similarly, $\# M_{\hat{B}}/\hat{\varphi}^\ast(M_{\hat{A}})$ 
divides $\#\ker(\hat{\varphi})$. Since $\ker(\hat{\varphi})\cong \Hom(\ker(\phi), \gm{K})$ (see \cite[Thm.1, p. 143]{MumfordAV}), 
we conclude that $\# M_{\hat{B}}/\hat{\varphi}^\ast(M_{\hat{A}})$ also divides $\#H$. Now one 
easily deduces from \eqref{eq2} the following:

\begin{lem}\label{lemPhil} Assume $A$ has semi-abelian reduction, and $\varphi: A\to B$ is an isogeny defined over $K$. 
If $\ell$ is a prime number which does not divide $\#\ker(\varphi)$, then $\varphi_\Phi$ induces an 
isomorphism $(\Phi_A)_\ell\cong (\Phi_B)_\ell$. 
\end{lem}

\begin{lem}\label{lemJNT2011} 
Let $K'$ be a finite unramified extension of $K$. 
Let $\varphi: A\to B$ be an isogeny defined over $K$ such that $H=\ker(\varphi)\subset A(K')$, i.e., 
$H$ becomes a constant group-scheme over $K'$.   
Let $H_0$ (resp. $H_1$) be the kernel (resp. image) of the homomorphism $H\to \Phi_A$ 
defined by \eqref{eqCanRedPhi}. Assume $A$ has toric reduction. Assume that 
either $\# H$ is coprime to the characteristic $p$ of $k$, or that $K$ has characteristic $0$ and its absolute ramification 
index is $< p-1$. Then there is an exact sequence 
$$
0\to H_1\to \Phi_A\xrightarrow{\varphi_\Phi} \Phi_B\to H_0\to 0.
$$
\end{lem}
\begin{proof}
Under these assumptions, we have $H\hookrightarrow \cA_k(k')$ and $H_0=\ker(\varphi_t)$. This implies 
$(M_A/\varphi^\ast(M_B))^\vee\cong H_0$. Next, \cite[Thm. 8.6]{CS}
implies that $M_{\hat{B}}/\hat{\varphi}^\ast(M_{\hat{A}})\cong H_1$.
Thus, we can rewrite \eqref{eq2} as
$$
0\to \ker(\varphi_\Phi)\to H_1\to H_0\to \coker(\varphi_\Phi)\to 0.
$$
Since $\ker(\varphi_\Phi)= H_1$, we conclude from this exact sequence that
$\coker(\varphi_\Phi)\cong H_0$.
\end{proof}

\section{Hecke Algebra}\label{sHA}

Since the $\Z$-algebra $\T$ is free of finite rank as a $\Z$-module, we 
can define the discriminant $\disc(\T)$ of $\T$ with respect to the trace pairing; cf. \cite[p. 66]{Reiner}. An algorithm 
for computing the discriminants of Hecke algebras is implemented in \texttt{Magma}; it gives $\disc(\T)=2^{11}\cdot 3$. 
We then obtain $$\T=\Z T_1 +\Z T_2+ \Z T_3 + \Z T_5 + \Z T_{11}$$ as a free $\Z$-module by comparing the 
discriminants. 
We have $\T\otimes_\Z\Q\cong \Q\times \Q(\sqrt{2})\times \Q(\sqrt{3})$. Let $$\tT= \Z\times \Z[\sqrt{2}]\times \Z[\sqrt{3}]$$ 
be the integral closure of $\T$ in $\T\otimes\Q$. 
Viewing $\T$ as an order in $\tT$, we have 
\begin{align}\label{eqExplHecke}
\nonumber T_1 & =(1, 1, 1)\\ 
 \nonumber T_2 &=(-1, -1+\sqrt{2}, \sqrt{3})\\
T_3 &=(-2, \sqrt{2}, 1-\sqrt{3})\\
\nonumber T_5 &=(-1, 1, -1)\\ 
\nonumber T_{11}&=( 2, 2-\sqrt{2}, -3+\sqrt{3}). 
\end{align}
One then observes that $\T=\Z v_1+ \Z v_2 +\Z v_3+\Z v_4+\Z v_5$, where 
$$
v_1 =(1, 1, 1), \quad v_2 =(0, 2, 0), \quad v_3 =(0, 0, 2), \quad v_4=(0, 2\sqrt{2}, 0), 
$$
$$
v_5=(-1, -1+\sqrt{2}, 2-\sqrt{3}),
$$
which implies
\begin{equation}\label{eqTrep}
\T\cong \left\{(a, b_1 + b_2\sqrt{2}, c_1+c_2\sqrt{3})\quad \bigg|\quad \begin{matrix} a,b_1, b_2, c_1, c_2\in \Z, \\ 
a\equiv b_1\equiv (c_1+c_2)\ \mod\ 2,\\ b_2\equiv c_2\ \mod\ 2 \end{matrix}  
\right\}. 
\end{equation}

Given a maximal ideal $\fm\lhd \T$, let $\T_\fm=\underset{\substack{\longleftarrow\\ n}}{\lim}\ \T/\fm^n$ denote the completion 
of $\T$ at $\fm$. 

\begin{prop}\label{propPrincipal}
Every maximal ideal in $\T$ of odd residue characteristic is principal. In particular, 
$\T_\fm$ is Gorenstein for any maximal ideal $\fm\lhd \T$ of odd residue characteristic; cf. \cite[p. 329]{Tilouine}. 
\end{prop}
\begin{proof} Since $$\disc(\T)=[\tT:\T]^2\cdot \disc(\tT)=[\tT:\T]^2 \cdot 2^5\cdot 3,$$ we get 
$[\tT:\T]=2^3$. Let $I_{\tT, 2'}$ be the set of ideals $I\lhd \tT$ such that $\tT/I$ is a finite ring of odd order. Let 
$I_{\T, 2'}$ be the set of ideals $I\lhd \T$ such that $\T/I$ is a finite ring of odd order. The 
argument of the proof of Proposition 7.20 in \cite{Cox} shows that the map $I\mapsto I\cap \T$ gives a bijection from $I_{\tT, 2'}$ 
to $I_{\T, 2'}$, with the inverse given by $I\mapsto I\tT$. Moreover, 
the proof of that proposition shows that for $I\in I_{\tT, 2'}$ we have $\tT/I\cong \T/I\cap \T$, so that 
this bijection restricts to a bijection between the maximal ideals of $\tT$ and $\T$ of odd residue 
characteristic. 

Since $\tT$ is a direct product of Euclidean domains, every ideal $I\in I_{\tT, 2'}$ is principal. Write $I=\theta\tT$. If $\theta\in \T$, 
then $I\cap \T= \theta\T$ is also principal, since $(\theta\T)\tT=\theta\tT$. Therefore, to prove the proposition it is 
enough to show that for every maximal ideal $\fm\in I_{\tT, 2'}$ we can choose a generator which lies in $\T$. 
Let $p>2$ be the residue characteristic of $\fm=\theta\tT$. If we write $\fm=\fm'\times \fm''\times \fm'''$, 
where $\fm'\lhd \Z$, $\fm''\lhd \Z[\sqrt{2}]$, $\fm'''\lhd \Z[\sqrt{3}]$, then 
one of these ideals is maximal of residue characteristic $p$, and the other two are equal to the corresponding ring. 
We consider three cases depending on which of the three ideals is proper. 

Case 1: $\fm'=p\Z$. Then $\theta=(p, 1,1)\in \T$. 

Case 2: $\fm''$ is proper. If $(p)$ is inert in $\Z[\sqrt{2}]$, then we can take $\theta=(1, p, 1)\in \T$. Now suppose 
$p=(\alpha+\beta\sqrt{2})(\alpha-\beta\sqrt{2})$ splits, where $\alpha, \beta\in \Z$. 
Note that $\alpha$ must be odd. If $\beta$ is even, then $\theta=(1, \alpha\pm \beta\sqrt{2}, 1)\in \T$. 
If $\beta$ is odd, then $\theta=(1, \alpha\pm \beta\sqrt{2}, 2+\sqrt{3})\in \T$, as $2+\sqrt{3}$ is a unit in $\Z[\sqrt{3}]$. 

Case 3: $\fm'''$ is proper. If $(p)$ is inert in $\Z[\sqrt{3}]$, then we can take $\theta=(1, 1, p)\in \T$. 
If $p=3$, then $\theta=(1, 1+\sqrt{2}, \sqrt{3})\in \T$, since $1+\sqrt{2}$ is a unit in $\Z[\sqrt{2}]$. Finally, suppose 
$p=(\alpha+\beta\sqrt{3})(\alpha-\beta\sqrt{3})$, where $\alpha, \beta\in \Z$. Considering $p=\alpha^2-3\beta^2$ 
modulo $2$, we get $1\equiv (\alpha+\beta)^2\ \mod\ 2$, so that $\alpha$ and $\beta$ have different parity.  
If $\alpha$ is odd and $\beta$ is even, then $\theta=(1,1, \alpha\pm\beta\sqrt{3})\in \T$. 
If $\alpha$ is even and $\beta$ is odd, then $\theta=(1,1+\sqrt{2}, \alpha\pm\beta\sqrt{3})\in \T$. 
\end{proof}

\begin{rem}
Let $\cO=\Z[i]$ be the Gaussian integers. Let $\cO'=\Z+3\cO=\Z+3i\Z$ be an order in $\cO$. We have $[\cO:\cO']=3$. 
The ideal $\fm=(2+i)\cO$ is maximal and $\cO/\fm\cong \F_5$. On the other hand, $\fm\cap \cO'=(5, 1+3i)\cO'$ 
is not principal, although $(5, 1+3i)\cO=\fm$. This indicates that 
Proposition \ref{propPrincipal} is not a special case of a general fact about orders. 
\end{rem}

\begin{defn}
The \textit{Eisenstein ideal} of $\T$ is the ideal $\cE\lhd \T$ generated by $T_\ell-(\ell+1)$ for all primes $\ell\nmid 65$. 
A maximal ideal $\fm\lhd \T$ in the support of the Eisenstein ideal is called an \textit{Eisenstein maximal ideal}. 
\end{defn}

\begin{prop}\label{propT/E} We have 
$$
\T/\cE \cong \Z/84\Z\cong \Z/4\Z\times \Z/3\Z\times \Z/7\Z.
$$  
\end{prop}
\begin{proof}
First, we explain how to compute the expansion of an arbitrary Hecke operator $T_m\in \T$ in terms of 
the $\Z$-basis $\{T_1, T_2, T_3, T_5, T_{11}\}$ of $\T$. Up to Galois conjugacy, there are three normalized 
$\T$-eigenforms in $S_2(65)$. The three coordinates of $T_m$ in the ring on the right hand-side of \eqref{eqTrep} 
are the eigenvalues with which $T_m$ acts on these eigenforms. 
Once we have this representation of $T_m$, thanks to \eqref{eqExplHecke}, 
finding the expansion of $T_m$ in terms of our basis amounts to solving a system of five linear equations in five variables. 
This strategy yields 
\begin{align*}
T_7 &= 2T_1-T_2-6T_3+9T_5-5T_{11},\\
T_{19} &= 2T_1+2T_2-4T_3+8T_5-3T_{11},\\
T_{29} &= -4T_1+T_2+12T_3-13T_5+9T_{11}.
\end{align*}

The Hecke operators $T_\ell$ for primes $\ell\nmid 65$ are all congruent to integers modulo $\cE$. 
Since $T_5=(T_7-T_{19})+3T_2+2T_3+2T_{11}$, we conclude that all Hecke operators are congruent 
to integers. Hence the natural map $\Z\to \T/\cE$ is surjective. We cannot have $\T/\cE=\Z$, for then 
there would exist a cusp form $f\in S_2(65)$ such that $T_\ell f=(\ell+1)f$, which would 
contradict the Ramanujan-Petersson bound. 
Therefore, $\T/\cE\cong \Z/n\Z$ for some integer $n$. 
Note that $T_5\equiv 29\ (\mod\ \cE)$. From the expansion of $T_7$, we obtain $168=2^3\cdot 3\cdot 7\equiv 0\ (\mod\ \cE)$;  
from the expansion of $T_{29}$, we obtain $252=2^2\cdot 3^2\cdot 7\equiv 0\ (\mod\ \cE)$; thus,  
$n$ divides $4\cdot 3\cdot 7=84$. 
On the other hand, the Eichler-Shimura congruence \cite[p. 89]{Mazur} implies that $\cE$ annihilates 
$J(\Q)_\tor\cong \Z/2\Z\times\Z/4\Z\times \Z/3\Z\times \Z/7\Z$; see Proposition \ref{propC=Tor}. 
Hence $n$ is divisible by the exponent of this group, which is $84$. 
\end{proof}

\begin{lem}
The Hecke operators $T_5$ and $T_{13}$ act on $\T/\cE\cong \Z/4\Z\times \Z/3\Z\times \Z/7\Z$ 
as $(1, -1, 1)$ and $(1, 1, -1)$, respectively. 
\end{lem}
\begin{proof}
In the proof of Proposition \ref{propT/E} we computed that $T_5\equiv 29\ (\mod\ \cE)$. 
Similarly, $T_{13} = -T_3+T_5-T_{11}\equiv 13\ (\mod\ \cE)$. From this the claim of the 
lemma immediately follows since, for example, $29 \equiv 1\ (\mod\ 4)$, $29 \equiv -1\ (\mod\ 3)$, and $29 \equiv 1\ (\mod\ 7)$. 
\end{proof}

\begin{rem} We note that $T_5$ and $T_{13}$ are actually equal to the negatives of 
the Atkin-Lehner involutions $W_5$ and $W_{13}$ acting on $S_2(65)$. The conclusion 
$(\T/\cE)_\odd\cong \Z/3\Z\times \Z/7\Z$ then can be deduced 
from Theorem 3.1.3 in \cite{Ohta}. 
\end{rem}

Proposition \ref{propT/E} implies that there are three Eisenstein maximal ideals in $\T$:   
\begin{align*}
\fm_2 & :=(\cE, 2) = (\cE, 2, T_5-1, T_{13}-1),\\ 
\fm_3 & :=(\cE, 3) = (\cE, 3, T_5+1, T_{13}-1),\\ 
\fm_7 & :=(\cE, 7) = (\cE, 7, T_5-1, T_{13}+1).
\end{align*}

\begin{prop}\label{propm2} We have:
\begin{itemize}
\item[(i)] The ideal $\fm_2\lhd \T$ is equal to the ideal 
$$
\left((2,1,1)\tT\right)\cap \T = \left\{(a, b_1 + b_2\sqrt{2}, c_1+c_2\sqrt{3})\in \T\ \big|\ a\in 2\Z\right\},  
$$
which is the unique maximal ideal of $\T$ of residue characteristic $2$. 
\item[(ii)] $\fm_2^n$ is not principal for any $n\geq 1$. 
\item[(iii)] $\T_{\fm_2}$ is not Gorenstein. 
\end{itemize}
\end{prop}
\begin{proof}
(i) The uniqueness of the maximal ideal of residue characteristic $2$ implies that it must be the Eisenstein maximal 
ideal $\fm_2$. To prove the uniqueness, note that 
each of the rings $\Z$, $\Z[\sqrt{2}]$, $\Z[\sqrt{3}]$ has a unique maximal ideal of residue characteristic $2$; these are  
generated by $2$, $\sqrt{2}$, and $1+\sqrt{3}$, respectively. One easily checks that 
$$
\fm:=((2,1,1)\tT)\cap \T =((1,\sqrt{2},1)\tT)\cap \T=((1,1,1+\sqrt{3})\tT)\cap \T,
$$
and $\T/\fm\cong \F_2$. 

(ii) To prove this statement it is enough to observe that $(1,0,0)\in \tT$ is in $\End_\T(\fm_2^n)$ but $(1,0,0)\not\in \T$. 

(iii) We apply \cite[Prop. 1.4 (iii)]{Tilouine}: Let $\overline{\fm}_2$ denote the image of $\fm_2$ in $\T/2\T$. 
Then $\T_{\fm_2}$ is Gorenstein if and only if $\dim_{\F_2}(\T/2\T)[\overline{\fm}_2]=1$. Note that 
$(2,0,0)$ and $(0,2,0)$ have distinct non-zero 
images in $\T/2\T$, since otherwise $(2,2,0)\in 2\T$, which would imply $(1,1,0)\in \T$. 
On the other hand, for any $\theta\in \fm_2$ we have $\theta(2,0,0)=(4a,0,0)=2(2a,0,0)\in 2\T$ for some $a\in \Z$. 
Therefore, $\overline{\fm}_2$ annihilates $(2,0,0)$, and similarly $\overline{\fm}_2$ 
annihilates $(0,2,0)$; thus, $\dim_{\F_2}(\T/2\T)[\overline{\fm}_2]\geq 2$. 
\end{proof}

$\Spec(\T)$ can be sketched as in Figure \ref{Fig1}. It has three irreducible components intersecting at $\fm_2$.  
The irreducible components containing the closed points $\fm_3$ and $\fm_7$ are determined by 
observing that $T_5+1=(0,2,0)$ and $T_5-1=(-2,0,-2)$, so $T_5$ acts as $-1$ (resp. $1$) on the component $\Spec(\Z[\sqrt{3}])$ 
(resp. $\Spec(\Z[\sqrt{2}])$). Finally, note that $\T_{\fm_7}\cong \Z_7$ and $\T_{\fm_3}\cong \Z_3[\sqrt{3}]$.

\begin{figure}
\begin{tikzpicture}[scale=0.5, inner sep=.5mm]
\draw[thick] (0,2) -- (8,6); 
\draw[thick] (0,3) -- (8,3); 
\draw[thick] (0,4) -- (8,0); 
\node at (6, 1) [circle,draw,fill=black][label=below:$\fm_3$]  {};
\node at (6, 3) [circle,draw,fill=black][label=below:$\fm_7$]  {};
\node at (2, 3) [circle,draw,fill=black][label=below:$\fm_2$] {};
\node at (8, 0) [label=right:$\Z{[\sqrt{3}]}$] {};
\node at (8, 3) [label=right:$\Z{[\sqrt{2}]}$] {};
\node at (8, 6) [label=right:$\Z$] {};
\end{tikzpicture}
\caption{$\Spec(\T)$}\label{Fig1}
\end{figure}


\section{Modular Jacobian}\label{sMJ}

There are exactly four cusps, denoted $[1]$, $[p]$, $[q]$ and $[pq]$, on $X_0(pq)$, where $p$ and $q$ are 
two distinct prime numbers. Let $\cC(pq)$ be the subgroup of $J_0(pq)$ generated 
by all cuspidal divisors. Since all cusps are $\Q$-rational, we have $\cC(pq)\subset J_0(pq)(\Q)$. 
Let $\Phi(p)$ and $\Phi(q)$ denote the component groups 
of $J_0(pq)$ at $p$ and $q$, and $\wp_p, \wp_q: \cC(pq)\to \Phi(p), \Phi(q)$ 
be the homomorphisms  induced by \eqref{eqCanRedPhi}.   

\begin{prop}
Let $p=5$ and $q=13$. 
Let $c_p$ and $c_q$ be the divisor classes of $[1]-[p]$ and $[1]-[q]$ in $J_0(pq)$. Denote $\cC:=\cC(pq)$. 
\begin{itemize}
\item[(i)] $\cC$ is generated by $c_p$ and $c_q$. The order of $c_p$ is $28$; the order of $c_q$ is $12$; the only relation between 
$c_p$ and $c_q$ in $\cC$ is $14 c_p = 6c_q$. This implies 
$$
\cC\cong \Z/2\Z\times\Z/4\Z\times \Z/3\Z\times \Z/7\Z. 
$$
\item[(ii)] $\Phi(p)\cong \Z/42\Z$ and $\Phi(q)\cong \Z/6\Z$. 
\item[(iii)] The order of $\wp_p(c_p)$ is $14$, and $\wp_p(c_q)=0$; this implies that there is an exact sequence
$$
0\to \langle c_q\rangle \to \cC \overset{\wp_p}{\To} \Phi(p)\to \Z/3\Z\to 0. 
$$
The order of $\wp_q(c_q)$ is $6$, and $\wp_q(c_p)=0$; this implies that there is an exact sequence
$$
0\to \langle c_p\rangle \to \cC \overset{\wp_q}{\To} \Phi(q)\to 0. 
$$
\end{itemize}
\end{prop}
\begin{proof}
(i) follows from \cite{CL}. The groups $\Phi(p)$ and $\Phi(q)$ can be computed 
from the structure of special fibres of $X_0(pq)$ using a well-known method of Raynaud; 
see \cite[p. 214]{Ogg} or the appendix in \cite{Mazur}. Finally, by considering the reductions of the cusps in 
the special fibre of the minimal regular model of $X_0(pq)$ over $\Z_p$, one 
can determine the homomorphism $\wp_p$ and $\wp_q$; cf. \cite[p. 1161]{PapikianJL}. 
\end{proof}

\begin{prop}\label{propC=Tor}
We have $\cC=J(\Q)_\tor$. 
\end{prop}
\begin{proof}
Obviously $\cC\subseteq J(\Q)_\tor$. On the other hand, $J$ has good reduction at any odd prime $p\nmid 65$, 
so by Proposition \ref{propKatz} we have an injective homomorphism $J(\Q)_\tor \hookrightarrow J(\F_p)$, 
where $J(\F_p)$ denotes the group of $\F_p$-rational points on the reduction of $J$ at $p$. The order of 
$J(\F_p)$ can be computed using \texttt{Magma}. 
We have $\# J(\F_3)=2^3\cdot 3^2\cdot 7$ and $\# J(\F_{11})=2^3\cdot 3\cdot 5\cdot 7^2\cdot 37$. 
Since the greatest common divisor of these numbers is $2^3\cdot 3\cdot 7= \#\cC$, the claim follows. 
\end{proof}

The Hecke ring $\T$ is isomorphic to a subring of endomorphisms of $J$ generated by the Hecke 
operators $T_n$ acting as correspondences on $X$. In fact, 
in our case $\T$ is the full ring of endomorphisms of $J$ 
(this can be proved as in \cite[Prop. 9.5]{Mazur}). For a maximal ideal $\fm\lhd \T$, we denote 
$$
J[\fm]=\bigcap_{\alpha\in \fm}\ker(J\xrightarrow{\alpha}J)
$$
Then $J[\fm]\subset J[p]$, where $p$ is the characteristic of $\T/\fm$. By a theorem of Mazur \cite[p. 341]{Tilouine}, $\T_\fm$ 
is Gorenstein if and only if $\dim_{\T/\fm} J[\fm]=2$. Therefore, using Proposition \ref{propPrincipal}, we conclude that $\dim_{\T/\fm} J[\fm]=2$ 
for any maximal ideal $\fm$ of odd residue characteristic. 

Let $p=3,7$ and $\fm_p$ be the corresponding Eisenstein maximal ideal. 
The Eichler-Shimura congruence relation implies that $\cE$ annihilates $J(\Q)_\tor=\cC$. 
Hence $\Z/p\Z\cong \cC_p\subset J[\fm_p]$. We have 
\begin{equation}\label{eqJ[m_p]}
0\To \Z/p\Z \To J[\fm_p] \To \mu_p\To 0,
\end{equation}
since $G_\Q$ acts on $\wedge^2 J[\fm_p]$ by the mod $p$ cyclotomic character; cf. \cite[p. 465]{RibetLL}. 
By \cite{LO}, the Shimura subgroup $\Sigma$ (= kernel of the functorial homomorphims $J_0(65)\to J_1(65)$) is 
\begin{equation}\label{eqSG}
\Sigma\cong \mu_2\times \mu_3, 
\end{equation}
and the Eisenstein ideal $\cE$ annihilates $\Sigma$. 
Therefore, \eqref{eqJ[m_p]} splits 
for $p=3$:
$$
J[\fm_3]= \cC_3\times \Sigma_3\cong \Z/3\Z\times \mu_3.
$$ 
\begin{lem}
The sequence \eqref{eqJ[m_p]} does not split for $p=7$. 
\end{lem}
\begin{proof} 
If \eqref{eqJ[m_p]} splits then $\Z/7\Z\times \Z/7\Z\subset J(\Q(\mu_7))_\tor$. 
Since $\ell=29$ splits completely in $\Q(\mu_7)$, by Proposition \ref{propKatz} we 
must have $7^2\mid \# J(\F_\ell)=2^3\cdot 3^2\cdot 7\cdot 13\cdot 23^2$. 
\end{proof}

\begin{rem}
Let $E$ be the elliptic curve defined by $y^2+xy=x^3-x$. It is easy to check that $E$ 
has a rational $2$-torsion point and $E[2]$ as a Galois module is a non-split extension 
$$
0\To \Z/2\Z\To E[2]\To \Z/2\Z\To 0.
$$ By Table 1 in \cite{Cremona}, $E$ is isomorphic 
to a subvariety of $J$. We claim that $E[2]\subset J[\fm_2]$. To see this, consider 
a Hecke operator $T_p=(a_p, b_p+\sqrt{2}c_p, d_p+\sqrt{3}e_p)$ for prime $p\nmid 65$, given as in \eqref{eqTrep}. 
$T_p$ acts on $E$ by multiplication by $a_p$. The fact that $\fm_2$ is Eisenstein implies 
that $a_p-(p+1)$ is even; thus, $T_p-(p+1)$ annihilates $E[2]$; thus $\fm_2=(2, \cE)$ annihilates $E[2]$.   
On the other hand, clearly $E[2]\not\subset \cC[2]$, as 
$\cC[2]$ is constant. Therefore, $\dim_{\T/\fm_2}J[\fm_2]\geq \dim_{\F_2}\cC[2]+1=3$. 
This gives a geometric proof of the fact that $\T_{\fm_2}$ is not Gorenstein. 
Note that Proposition \ref{propC=Tor} implies that $\Sigma[2]\subset \cC[2]$, since $\mu_2\cong \Z/2\Z$ 
is constant over $\Q$.
\end{rem}

\begin{prop}\label{propNoH}
Let $\fm\lhd \T$ be an Eisenstein maximal ideal of odd residue characteristic $p$. Let $H\subset J[\fm^s]$, $s\geq 1$, 
be a $\T[G_\Q]$-module. If $J[\fm]\not\subset H$, then $H\subsetneq J[\fm]$. 
\end{prop}
\begin{proof}
We will assume that 
$J[\fm]\not \subset H$ and $H\not\subset J[\fm]$, and reach a contradiction. First, we make some simplifications. 
Since $H[\fm^2]\subset J[\fm^2]$ is a $\T[G_\Q]$-module satisfying the same assumptions, if we want to 
show that $H$ does not exist, it is enough to prove the non-existence under the additional assumption that $H\subset J[\fm^2]$. 

\begin{lem}
We have $H\cong \T/\fm^2$. 
\end{lem}
\begin{proof} We can consider $H$ as a finite $\T_\fm$-module. Since $\T_\fm$ 
is a DVR, we have 
$$
H\cong \T_\fm/\fm^{s_1}\times \cdots\times \T_\fm/\fm^{s_r}\cong \T/\fm^{s_1}\times \cdots\times \T/\fm^{s_r}
$$
for some $1\leq s_1\leq s_2\leq \cdots \leq s_r\leq 2$. 
Since $\dim_{\T/\fm}J[\fm]=2$, and $H[\fm]\cong (\T/\fm)^r\subsetneq J[\fm]$, 
we must have $r= 1$, i.e.,  
$H\cong \T/\fm^{s}$ for $s=1$ or $s=2$. 
If $s=1$, then $H\subset J[\fm]$, contrary to our assumption, so $s=2$.  
\end{proof}

Note that 
$$
\T/\fm^2\cong 
\begin{cases}
\Z/p^2\Z & \text{if }p=7;\\
\F_p[x]/(x^2) & \text{if }p=3. 
\end{cases}
$$
Let $K:=\Q(H)$. If $K=\Q$, then $p^2= \# H$ divides $\# J(\Q)_\tor$. This 
contradicts Proposition \ref{propC=Tor}, so we will assume 
from now on that $K\neq \Q$. Let $\eta$ be a generator of $\fm$. 
Note that $\eta H=H[\eta]\subset J[\fm]$ is a proper non-trivial 
Galois invariant subgroup. On the other hand, the $G_\Q$-invariant subgroups of $J[\fm]$ 
are $\Z/p\Z$ and $\mu_p$, so either 
\begin{equation}\label{eqHseq}
0\to \Z/p\Z\to H\xrightarrow{\eta}\Z/p\Z\to 0, 
\end{equation}
or 
\begin{equation}\label{eqHseq2}
0\to \mu_p\to H\xrightarrow{\eta}\mu_p\to 0. 
\end{equation}
Moreover, the second possibility does not occur for $p=7$, since \eqref{eqJ[m_p]} does not split.

\begin{lem}\label{lem3.6} Let $K_p$ denote the unique degree $p$ extension of $\Q$ contained in $\Q(\mu_{p^2})$. 
\begin{enumerate}
\item If $p=7$, then $K=K_p$. 
\item Assume $p=3$. In case of \eqref{eqHseq}, we have $[K:\Q]=p$ and $K\subset K_p\Q(\mu_{13})$. In case of 
\eqref{eqHseq2}, we have $\Q(\mu_p)\subseteq K\subset \Q(\mu_{p^2}, \mu_{13})$.  
\end{enumerate}
\end{lem}
\begin{proof} Since the actions of $\T$ and $G_\Q$ on $H$ commute, we have  
$$\Gal(K/\Q)\subset \Aut_\T(\T/\fm^2)\cong (\T/\fm^2)^\times\cong \Z/(p-1)p\Z.$$ 
Hence $K/\Q$ is an abelian extension. 
Since $J$ has good reduction away from $5$ and $13$, the 
extension $K/\Q$ is unramified away from $p, 5, 13$. By class field theory, 
$K$ is a subfield of a cyclotomic extension $\Q(\mu_{p^{n_1}}, \mu_{5^{n_2}}, \mu_{13^{n_3}})$, for some $n_1, n_2, n_3\geq 1$. 
We have 
\begin{align*}
&\Gal(\Q(\mu_{p^{n_1}}, \mu_{5^{n_2}}, \mu_{13^{n_3}})/\Q)\\  &\cong \Gal(\Q(\mu_{p^{n_1}}/\Q)\times 
\Gal(\Q(\mu_{5^{n_2}}/\Q)\times \Gal(\Q(\mu_{13^{n_3}}/\Q) \\ 
&\cong \Z/p^{n_1-1}(p-1)\Z \times \Z/5^{n_2-1}(5-1)\Z \times \Z/13^{n_3-1}(13-1)\Z. 
\end{align*}

Assume $p=7$. Since in this case $H$ is as in \eqref{eqHseq}, 
$G_\Q$ acts trivially on $pH$, so $\Gal(K/\Q)$ is in the subgroup of units $(\Z/p^2\Z)^\times$
which satisfy $ap\equiv p\ (\mod\ p^2)$, or equivalently, $a\equiv 1 \ (\mod\ p)$. The units with this property 
form the cyclic subgroup of order $p$ in $(\Z/p^2\Z)^\times$. Hence $K/\Q$ is an abelian extension of degree $p$. 
Since $p$ does not divide $(5-1)5^{n_2-1}$ or $(13-1)13^{n_3-1}$, the field $K$ is fixed by $\Gal(\Q(\mu_{5^{n_2}})/\Q)\times \Gal(\Q(\mu_{13^{n_3}})/\Q)$. 
Therefore, $K\subset \Q(\mu_{p^{n_1}})$ is a subfield of degree $p$ over $\Q$. There is a unique such 
field (as $\Gal(\Q(\mu_{p^{n_1}}/\Q)$ is cyclic), and it is contained in $\Q(\mu_{p^2})$. 

Assume $p=3$ and $H$ fits into an exact sequence \eqref{eqHseq}. 
By the argument in the previous paragraph, $[K:\Q]=p$. 
Let $F:=\Q(\mu_{13})$ and $K'=F(H)$. We know that $[K':F]=1$ or $p$. 
Note that $$\Gal(\Q(\mu_{p^{n_1}}, \mu_{5^{n_2}},\mu_{13^{n_3}})/F)\cong \Z/(p-1)p^{n_1-1}\times \Z(5-1)5^{n_2-1}\times \Z/13^{n_3-1}\Z, $$   
so as in the case of $p=7$, we get $F(H)\subset K_pF$. 
 
 Finally, assume $p=3$ and $H$ fits into an exact sequence \eqref{eqHseq2}. Then obviously $\Q(\mu_p)\subset K$. 
 Over $L:=\Q(\mu_p)$, the group scheme $H$ fits into an exact sequence \eqref{eqHseq}, so, as in the earlier cases, $L(H)/L$  
 is cyclic of order $1$ or $p$. If $H$ is not constant over $FL$, then $[FL(H):FL]=p$. 
On the other hand,  
$$\Gal(\Q(\mu_{p^{n_1}}, \mu_{5^{n_2}},\mu_{13^{n_3}})/FL)\cong \Z/p^{n_1-1}\times \Z(5-1)5^{n_2-1}\times \Z/13^{n_3-1}\Z.$$ 
As in the earlier cases, this implies that $FL(H)\subset K_pFL=\Q(\mu_{p^2}, \mu_{13})$. 
Overall, we see that $K$ is always a subfield of $\Q(\mu_{p^2}, \mu_{13})$. 
\end{proof}

Assume $p=7$. By Lemma \ref{lem3.6}, we have $K=K_p$. 
Let $\ell$ be a prime which splits completely in $K_p$. Then $H$ 
is constant over $\Q_\ell$, so $H\subset J(\Q_\ell)_\tor$. On the other hand, under the canonical reduction map, we have an injection 
$J(\Q_\ell)_\tor\hookrightarrow J(\F_\ell)$; see Proposition \ref{propKatz}. Therefore, we must have $p^2\mid \# J(\F_\ell)$. 
It is easy to show that a prime $\ell$ splits completely in $K_p$ if and only if its order in $(\Z/p^2\Z)^\times$ is coprime to $p$. 
We can take $3$ as a generator of $(\Z/p^2\Z)^\times$. The elements of orders coprime to $p$ 
are the powers of $3^7\equiv 31$. These are $\{31, 30, 48, 18, 19, 1\}$. Thus, the smallest prime 
that splits completely in $K_7$ is $19$, and $\# J(\F_{19})=2^3\cdot 3^2\cdot 7\cdot 13\cdot 23^2$. 
As $7^2$ does not divide this number, we get a contradiction.  

Assume $p=3$. By Lemma \ref{lem3.6}, we have $\Q(H)\subset \Q(\mu_{13}, \mu_{p^2})$. 
Since $\mu_p$ is constant over $K'$, 
we have $\Z/p\Z\times \Z/p\Z\cong J(K')[\fm]\subset J(K')_\tor\subset J(\Q_\ell)$. 
Since $H$ is also constant over $K'$, we also have $\Z/p\Z\times \Z/p\Z\cong H\subset J(\Q_\ell)$. 
Since $J[\fm]\not\subset H$, we see that $J(\Q_\ell)$ contains a subgroup isomorphic to $(\Z/p\Z)^3$. 
As earlier, this implies that $p^3\mid \#J(\F_{\ell})$.
A prime $\ell$ splits completely in $K':=\Q(\mu_{13}, \mu_{p^2})$ if and only if 
$\ell\equiv 1\ (\mod\ 9)$ and $\ell\equiv 1\ (\mod\ 13)$. The smallest such prime is 
$\ell=937$, and  
$\# J(\F_{937})= 2^{13}\cdot 3^2\cdot 7\cdot 11^2\cdot 41\cdot 97\cdot 2963$. 
As $3^3$ does not divide this number, we get a contradiction. 
This concludes the proof of Proposition \ref{propNoH}.
\end{proof}

Let $A$ be an abelian variety over $\Q$ and $\pi: J\to A$ an isogeny defined over $\Q$. 
Assume $\ker(\pi)$ is invariant under the action of $\T$, i.e., $\ker(\pi)$ is a finite $\T[G_\Q]$-module. 
We can decompose $\ker(\pi)=\ker(\pi)_2\times \ker(\pi)_\odd$; each of these subgroups 
is also a $\T[G_\Q]$-module.  Let the maximal ideal $\fm\lhd \T$ be in the support of 
$H:=\ker(\pi)_\odd$. Since $\fm$ has odd residue characteristic, $\fm=\eta\T$ 
is principal by Proposition \ref{propPrincipal}. If $\ker(\eta)=J[\fm]\subset H$, then we can decompose 
$\pi=\pi'\circ \eta$, where $\pi': J\to A$ is another isogeny whose kernel is a $\T[G_\Q]$-module but 
with smaller odd component than $\pi$. We can apply the same argument to $\pi'$ and continue this process 
until we obtain an isogeny whose kernel does not contain any $J[\fm]$ with $\fm$ having odd residue 
characteristic. From now on we assume that $\pi$ itself has this property.  

Since $\fm$ has odd residue characteristic, 
the $\T[G_\Q]$-module $J[\fm]$ is $2$-dimensional over $\T/\fm$.  
By \cite[Prop. 14.2]{Mazur} and \cite[Thm. 5.2]{RibetLL}, if $\fm$ is not Eisenstein, then $J[\fm]$ 
is irreducible. Since $J[\fm]\cap H\neq 0$, we must have $J[\fm]\subset H$, which contradicts 
our assumption on $\pi$. Hence $H$ is supported on the Eisenstein maximal ideals $\fm_3$ and $\fm_7$. 
We decompose $H=H_3\times H_7$ into $3$-primary and $7$-primary components, which 
themselves are $\T[G_\Q]$-modules. Now $H_p\subset J[\fm_p^s]$ for some $s\geq 1$, $p=3,7$, 
and $J[\fm_p]\not\subset H_p$. Applying Proposition \ref{propNoH}, we conclude that $H_p\subsetneq J[\fm_p]$. 
Thus $H_7=0$ or $\cC_7$, and $H_3=0$ or $\Sigma_3$ or $\cC_3$. Overall, $H$ 
can be one of the following subgroups of $J$:
\begin{equation}\label{eqHlist}
0, \quad \cC_3, \quad \Sigma_3, \quad \cC_7, \quad  \cC_3\times \cC_7, \quad \Sigma_3\times \cC_7. 
\end{equation}

\begin{thm}\label{thmMain} If $A=J'$, then for $\pi: J\to J'$ chosen with the minimality condition discussed above, we must have 
$H=\cC_7$. 
\end{thm}
\begin{proof}
The reductions of $J$ and $J'$ at $p=5$ or $13$ are purely toric, 
cf. \cite{Ogg}, \cite{RibetLL}.  
Let $\Phi(5)'$ and $\Phi(13)'$ be the component groups of $J'$ at $5$ and $13$. We have 
(see \cite[p. 214]{Ogg}):
$$
\Phi(5)'\cong \Z/6\Z, \qquad \Phi(13)'\cong \Z/42\Z. 
$$

We decompose $\pi: J\to J'$ as $J\to J/H\overset{\pi'}{\To} J'$, 
where $\ker(\pi')$ is isomorphic to the $2$-primary part of $\ker(\pi)$. 
Let $\Phi(p)''$ be the component group of $J/H$ at $p$. 
By Lemma \ref{lemPhil} we must have $(\Phi(p)'')_\odd\cong (\Phi(p)')_\odd$. On the other hand, since we know 
the image and kernel of $\wp_p:\cC\to \Phi(p)$, we 
can compute $\# (\Phi(p)'')_\odd$ for each possible $H$ from the list \eqref{eqHlist} using Lemma \ref{lemJNT2011}. 
This simple calculation shows that the only possible $H$ is $\cC_7$. 
(Note that the group scheme $\Sigma_3$ becomes constant over an unramified extension of $\Q_p$, but  
it is not important to know whether 
$\wp_p: \Sigma_3\to \Phi(p)$ is injective or trivial; neither of these possibilities gives the correct $\Phi(p)''$ 
if $\Sigma_3\subset H$.)
\end{proof}

\begin{rem}\label{rem4.9}
Let $N=5\cdot 7$. In this case, 
\begin{align*}
\T=\Z[T_3]& \cong \Z[x]/(x-1)(x^2+x-4)\\  
&\cong \{(a, b+c\alpha)\in \Z\times \Z[\alpha]\ \big|\ a, b, c\in \Z,\ a\equiv b+c\ (\mod\ 2)\},
\end{align*}
where $\alpha:=-\frac{1+\sqrt{17}}{2}$. Note that $\Z[\alpha]$ 
is the ring of integers in $\Q(\sqrt{17})$, and $\Z[\alpha]$ is a Euclidean domain with respect to 
the usual norm. We have 
$$
\cC\cong \Z/2\Z\times\Z/8\Z\times \Z/3\Z,\qquad \Sigma\cong \mu_4\times \mu_3. 
$$
There is a unique Eisenstein maximal ideal $\fm_3\lhd \T$ of odd residue characteristic. 
There is a unique $\Q$-isogeny class of elliptic curves 
of level $35$. The optimal curve is \cite[p. 112]{Cremona}
$$
E: y^2 +  y = x^3 +x^2 + 9x +1. 
$$
We have $E[3]\cong \mu_3\times \Z/3\Z$.  
Since $\T_\fm$ is Gorenstein for any maximal ideal $\fm\lhd \T$ (as $\T$ is monogenic), $J[\fm]$ 
is two dimensional over $\T/\fm$, so $J[\fm_3]=E[3]= \cC_3\times \Sigma_3$. Now 
it is easy to analyze all $\T[G_\Q]$-submodules of $J$ supported on $\fm_3$. An argument similar 
to the argument of the proof of Theorem \ref{thmMain} then implies that there is a Ribet isogeny $\pi: J\to J'$ 
with $\ker(\pi)_\odd=0$. Ogg's conjecture in this case predicts that $\ker(\pi)\cong \Z/2\Z\subset \cC_2$. 
\end{rem}

\begin{rem}\label{rem4.10}
Let $N=3\cdot 13$. In this case, 
\begin{align*}
\T=\Z[T_2]& \cong \Z[x]/(x-1)(x^2+2x-1)\\  
&\cong \{(a, b+c\sqrt{2})\in \Z\times \Z[\sqrt{2}]\ \big|\ a, b, c\in \Z,\ a\equiv b\ (\mod\ 2)\},
\end{align*}
We have 
$$
\cC\cong \Z/2\Z\times\Z/4\Z\times \Z/7\Z,\qquad \Sigma\cong \mu_4. 
$$
There is a unique Eisenstein maximal ideal $\fm_7\lhd \T$ of odd residue characteristic. 
$J[\fm]$ fits into the exact sequence \eqref{eqJ[m_p]}, which is non-split in this case.  
One can classify $\T[G_\Q]$-submodules of $J$ supported on $\fm_7$ using an argument similar to 
the argument we used in Proposition \ref{propNoH}. Finally, one deduces as in 
Theorem \ref{thmMain} that there is a Ribet isogeny $\pi: J\to J'$ 
with $\ker(\pi)_\odd=\cC_7\cong \Z/7\Z$. Ogg's conjecture in this case predicts that $\ker(\pi)=\cC_7$. 
\end{rem}

\subsection*{Acknowledgements} This work was carried out in part while the second  
 author was visiting the Taida Institute for Mathematical Sciences in Taipei and 
 the Max Planck Institute for Mathematics in Bonn in 2016. 
 He thanks these institutes for their hospitality, excellent working conditions, and financial support. 
 He is also grateful to Fu-Tsun Wei for very useful discussions related to the topic of this paper.



\begin{thebibliography}{10}

\bibitem{NM}
S.~Bosch, W.~L{\"u}tkebohmert, and M.~Raynaud, \emph{{N{\'e}}ron models},
  Springer-Verlag, 1990.

\bibitem{CL}
S.-K. Chua and S.~Ling, \emph{On the rational cuspidal subgroup and the
  rational torsion points of {$J_0(pq)$}}, Proc. Amer. Math. Soc. \textbf{125}
  (1997), 2255--2263.

\bibitem{CS}
B.~Conrad and W.~Stein, \emph{Component groups of purely toric quotients},
  Math. Res. Lett. \textbf{8} (2001), 745--766.

\bibitem{Cox}
D.~Cox, \emph{Primes of the form {$x^2 + ny^2$}}, second ed., Pure and Applied
  Mathematics (Hoboken), John Wiley \& Sons, Inc., Hoboken, NJ, 2013, Fermat,
  class field theory, and complex multiplication.

\bibitem{Cremona}
J.~Cremona, \emph{Algorithms for modular elliptic curves}, second ed.,
  Cambridge University Press, Cambridge, 1997.

\bibitem{GM}
J.~Gonz\'alez and S.~Molina, \emph{The kernel of {R}ibet's isogeny for genus
  three {S}himura curves}, J. Math. Soc. Japan \textbf{68} (2016), no.~2,
  609--635.

\bibitem{GoRo}
J.~Gonz{\'a}lez and V.~Rotger, \emph{Equations of {S}himura curves of genus
  two}, Int. Math. Res. Not. (2004), no.~14, 661--674.

\bibitem{SGA7}
A.~Grothendieck, \emph{Mod{\`e}les de {N{\'e}}ron et monodromie}, SGA 7,
  Expos{\'e} IX, 1972.

\bibitem{Helm}
D.~Helm, \emph{On maps between modular {J}acobians and {J}acobians of {S}himura
  curves}, Israel J. Math. \textbf{160} (2007), 61--117.

\bibitem{JordanLivne}
B.~Jordan and R.~Livn\'e, \emph{On the {N}\'eron model of {J}acobians of
  {S}himura curves}, Compositio Math. \textbf{60} (1986), no.~2, 227--236.

\bibitem{Katz}
N.~Katz, \emph{Galois properties of torsion points on abelian varieties},
  Invent. Math. \textbf{62} (1981), no.~3, 481--502.

\bibitem{LO}
S.~Ling and J.~Oesterl{\'e}, \emph{The {S}himura subgroup of {$J_0(N)$}},
  Ast\'erisque (1991), no.~196-197, 171--203, Courbes modulaires et courbes de
  Shimura (Orsay, 1987/1988).

\bibitem{Mazur}
B.~Mazur, \emph{Modular curves and the {E}isenstein ideal}, Inst. Hautes
  \'Etudes Sci. Publ. Math. \textbf{47} (1977), 33--186.

\bibitem{Michon}
J.-F. Michon, \emph{Courbes de {S}himura hyperelliptiques}, Bull. Soc. Math.
  France \textbf{109} (1981), no.~2, 217--225.

\bibitem{MumfordAV}
D.~Mumford, \emph{Abelian varieties}, Tata Institute of Fundamental Research
  Studies in Mathematics, No. 5, Published for the Tata Institute of
  Fundamental Research, Bombay; Oxford University Press, London, 1970.

\bibitem{Ogg}
A.~Ogg, \emph{Mauvaise r\'eduction des courbes de {S}himura}, S\'eminaire de
  th\'eorie des nombres, {P}aris 1983--84, Progr. Math., vol.~59, Birkh\"auser,
  1985, pp.~199--217.

\bibitem{Ohta}
M.~Ohta, \emph{Eisenstein ideals and the rational torsion subgroups of modular
  {J}acobian varieties {II}}, Tokyo J. Math. \textbf{37} (2014), no.~2,
  273--318.

\bibitem{PapikianJL}
M.~Papikian, \emph{On {J}acquet-{L}anglands isogeny over function fields}, J.
  Number Theory \textbf{131} (2011), no.~7, 1149--1175.

\bibitem{Reiner}
I.~Reiner, \emph{Maximal orders}, London Mathematical Society Monographs. New
  Series, vol.~28, The Clarendon Press, Oxford University Press, Oxford, 2003,
  Corrected reprint of the 1975 original, With a foreword by M. J. Taylor.

\bibitem{RibetIsogeny}
K.~Ribet, \emph{Sur les vari\'et\'es ab\'eliennes \`a multiplications
  r\'eelles}, C. R. Acad. Sci. Paris S\'er. A-B \textbf{291} (1980),
  A121--A123.

\bibitem{RibetSTN}
K.~Ribet, \emph{On the component groups and the {S}himura subgroup of
  {$J_0(N)$}}, S\'eminaire de {T}h\'eorie des {N}ombres, 1987--1988 ({T}alence,
  1987--1988), Univ. Bordeaux I, Talence, 1987/1988, pp.~Exp.\ No.\ 6, 10.

\bibitem{RibetLL}
K.~Ribet, \emph{On modular representations of {$\Gal(\overline{\Q}/{\Q})$}
  arising from modular forms}, Invent. Math. \textbf{100} (1990), 431--476.

\bibitem{Tilouine}
J.~Tilouine, \emph{Hecke algebras and the {G}orenstein property}, Modular forms
  and {F}ermat's last theorem ({B}oston, {MA}, 1995), Springer, New York, 1997,
  pp.~327--342.

\bibitem{YooBLMS}
H.~Yoo, \emph{Rational torsion points on {J}acobians of {S}himura curves},
  Bull. Lond. Math. Soc. \textbf{48} (2016), no.~1, 163--171.

\end{thebibliography}


\end{document}